\documentclass[leqno, amstex,12pt, amssymb]{article}

\usepackage{mathtext}
\usepackage[cp1251]{inputenc}
\usepackage[T2A]{fontenc}
\usepackage[dvips]{graphicx}
\usepackage{amsmath}
\usepackage{amssymb}
\usepackage{amsxtra}
\usepackage{latexsym}
\usepackage{ifthen}
\date{}

\newcounter{lemma}

\newcounter{corollary}

\newcounter{remark}

\newcounter{theorem}

\newcounter{proposition}

\newcounter{example}

\begin{document}

\markboth{D.~ROMASH, E.~SEVOST'YANOV}{\centerline{On uniformly
lightness ...}}

\def\cc{\setcounter{equation}{0}
\setcounter{figure}{0}\setcounter{table}{0}}

\overfullrule=0pt


\title{{\bf On uniformly lightness of one class of mappings and Koebe-Bloch theorem}}

\author{D.~ROMASH$^1$\footnote{dromash@num8erz.eu}, E.~SEVOST'YANOV$^{1,
2}$\footnote{Corresponding author, esevostyanov2009@gmail.com}}

\maketitle

{\small $^{1}$Zhytomyr Ivan Franko State University, 40, Velyka
Berdychivs'ka Str., 10 008  Zhytomyr, UKRAINE

$^{2}$Institute of Applied Mathematics and Mechanics of NAS of
Ukraine, 19 Henerala Batyuka Str., 84 116 Slov'yans'k, UKRAINE


\medskip

\maketitle

\medskip
{\bf Abstract.} We consider mappings satisfying a certain estimate
of the distortion of the modulus of families of paths, similar to
the geometric definition of quasiconformal mappings. Under
appropriate restrictions, we show that the class of such mappings is
uniformly light, i.e., the chordal diameter of the image of continua
whose diameter is bounded below is also bounded below uniformly over
the class. This result holds for any domain $D$ and for continua $C$
lying in the fixed compactum $K$ in this domain. If $D$ has a
special geometry, this result may be true in the general case, i.e.,
for any continua $C$ in $D.$


\section{Introduction}
The paper is devoted to mappings with bounded and finite distortion,
see \cite{AFW}, \cite{Cr$_1$}--\cite{Cr$_2$}, \cite{MRSY}, \cite{M},
\cite{PSS}, \cite{Ri}, \cite{RV} and \cite{SalSt}. Recall that a
mapping $f:D\rightarrow {\Bbb R}^n$ is called {\it discrete} if the
pre-image $\{f^{-1}\left(y\right)\}$ of each point $y\,\in\,{\Bbb
R}^n$ consists of isolated points, and {\it is open} if the image of
any open set $U\subset D$ is an open set in ${\Bbb R}^n.$ A mapping
$f:D\rightarrow \overline{{\Bbb R}^n}$ is said to be {\it light} if
${\rm dim\,}\{f^{\,-1}(y)\}=0$ for every $y\in \overline{{\Bbb
R}^n}.$ It is well known that $K$-quasiregular mappings are discrete
and open, see, for example~\cite[Theorem~I.4.1]{Ri}. Observe that,
all quasiregular mappings $f:D\rightarrow {\Bbb R}^n$ satisfy the
condition
\begin{equation*}\label{eq2}
M(\Gamma)\leqslant K N(f, D) M(f(\Gamma))
\end{equation*}
for any family $\Gamma$ of paths $\gamma$ in a domain $D,$ where $M$
is a conformal modulus of families of paths, $K={\rm ess \sup}\,
K_O(x, f),$
\begin{gather*}K_{O}(x,f)\quad =\quad \left\{
\begin{array}{rr}
\frac{\Vert f^{\,\prime}(x)\Vert^n}{|J(x,f)|}, & J(x,f)\ne 0,\\
1, & f^{\,\prime}(x)=0, \\
\infty, & {\rm in\,other\,cases}
\end{array}
\right.\,,\\
\Vert f^{\,\prime}(x)\Vert\,=\,\max\limits_{h\in {\Bbb R}^n
\backslash \{0\}} \frac {|f^{\,\prime}(x)h|}{|h|}\,,\quad
J(x,f)=\det f^{\,\prime}(x)\,,\\
N(y, f, D)\,=\,{\rm card}\,\left\{x\in D: f(x)=y\right\}\,, N(f,
D)\,=\,\sup\limits_{y\in{\Bbb R}^n}\,N(y, f, D)\,, \end{gather*}
see, e.g., \cite[Theorem~6.7.II]{Ri}. Let us now pose the question
of the openness and discreteness of the mapping $f:D\rightarrow
{\Bbb R}^n,$ $n\geqslant 2,$ satisfying the condition
\begin{equation} \label{eq2*B}
M(\Gamma )\leqslant \int\limits_{f(D)} Q(y)\cdot \rho_*^n (y) dm(y)
\end{equation}
for every $\rho_*\in {\rm adm}\,f(\Gamma)$ with respect to the
conformal modulus $M(\Gamma):=M_n(\Gamma)$ and a given function
$Q:{\Bbb R}^n\rightarrow [0, \infty].$ Note that even under
relatively good conditions on the function $Q,$ the mapping $f$ is,
generally speaking, neither discrete nor open. For example, even if
$Q$ is equal to the identical constant, this is generally not the
case (see the examples given in~\cite[Section~10]{Sev$_2$}).
Nevertheless, the mapping~$f$ in this case is light; see
\cite[Theorem~1]{SevSkv}, cf.~\cite[Theorem~10.1]{Sev$_3$},
\cite[Theorem~1.1]{Cr$_1$}. The goal of this paper is to obtain an
even more nontrivial result, namely, to show that the family of
mappings in~(\ref{eq2*B}) is ``light in totality,'' i.e., {\it it
cannot arbitrarily compress the family of continua under appropriate
conditions on the function $Q$ and the geometry of the given
domain.} For more detailed explanations and a statement of the
problem, we turn to the definitions below.

\medskip
Below $dm(x)$ denotes the element of the Lebesgue measure in ${\Bbb
R}^n.$ Everywhere further the boundary $\partial A $ of the set $A$
and the closure $\overline{A}$ should be understood in the sense of
the extended Euclidean space $\overline{{\Bbb R}^n}.$ Recall that, a
Borel function $\rho:{\Bbb R}^n\,\rightarrow [0,\infty] $ is called
{\it admissible} for the family $\Gamma$ of paths $\gamma$ in ${\Bbb
R}^n,$ if the relation
\begin{equation*}\label{eq1.4}
\int\limits_{\gamma}\rho (x)\, |dx|\geqslant 1
\end{equation*}
holds for all (locally rectifiable) paths $ \gamma \in \Gamma.$ In
this case, we write: $\rho \in {\rm adm} \,\Gamma .$  Given
$p\geqslant 1,$  the {\it $p$-modulus} of $\Gamma $ is defined by
the equality
\begin{equation*}\label{eq1.3gl0}
M_p(\Gamma)=\inf\limits_{\rho \in \,{\rm adm}\,\Gamma}
\int\limits_{{\Bbb R}^n} \rho^p (x)\,dm(x)\,.
\end{equation*}
Given $x_0\in{\Bbb R}^n,$ we put
\begin{gather*}
B(x_0, r)=\{x\in {\Bbb R}^n: |x-x_0|<r\}\,, \quad {\Bbb B}^n=B(0,
1)\,,\\
S(x_0,r) = \{ x\,\in\,{\Bbb R}^n : |x-x_0|=r\}\,. \end{gather*}
Let $y_0\in {\Bbb R}^n,$ $0<r_1<r_2<\infty$ and
\begin{equation}\label{eq1**}
A=A(y_0, r_1,r_2)=\left\{ y\,\in\,{\Bbb R}^n:
r_1<|y-y_0|<r_2\right\}\,.\end{equation}
Given sets $E,$ $F\subset\overline{{\Bbb R}^n}$ and a domain
$D\subset {\Bbb R}^n$ we denote by $\Gamma(E,F,D)$ a family of all
paths $\gamma:[a,b]\rightarrow \overline{{\Bbb R}^n}$ such that
$\gamma(a)\in E,\gamma(b)\in\,F $ and $\gamma(t)\in D$ for $t \in
(a, b).$ If $f:D\rightarrow {\Bbb R}^n,$ $y_0\in {\Bbb R}^n$ and
$0<r_1<r_2<d_0=\sup\limits_{y\in f(D)}|y-y_0|,$ then by
$\Gamma_f(y_0, r_1, r_2)$ we denote the family of all paths $\gamma$
in $D$ such that $f(\gamma)\in \Gamma(S(y_0, r_1), S(y_0, r_2),
A(y_0,r_1,r_2)).$ Let $Q:{\Bbb R}^n\rightarrow [0, \infty]$ be a
Lebesgue measurable function. We say that {\it $f$ satisfies
Poletsky inverse inequality} at the point $y_0\in {\Bbb R}^n$ with
respect to $p$-modulus, if the relation
\begin{equation}\label{eq2*A}
M_p(\Gamma_f(y_0, r_1, r_2))\leqslant
\int\limits_{A(y_0,r_1,r_2)\cap f(D)} Q(y)\cdot \eta^p (|y-y_0|)\,
dm(y)
\end{equation}
holds for any Lebesgue measurable function $\eta:
(r_1,r_2)\rightarrow [0,\infty ]$ such that
\begin{equation}\label{eqA2}
\int\limits_{r_1}^{r_2}\eta(r)\, dr\geqslant 1\,.
\end{equation}
We say that $f$ satisfies Poletsky inverse inequality at the point
$y_0=\infty $ with respect to $p$-modulus, if the
relation~(\ref{eq2*A}) holds for $y_0=0$ with
$\widetilde{Q}(y)=Q\left(\frac{y}{|y|^2}\right).$ The examples of
mappings satisfying relations~(\ref{eq2*A})--(\ref{eqA2}) are
classes of quasiconformal mappings, as well as quasiregular mappings
with finite multiplicity (see e.g. \cite[Definition~13.1]{Va},
\cite[Remark~2.5.II]{Ri}). We set
\begin{equation*}\label{eq12}
q_{y_0}(r)=\frac{1}{\omega_{n-1}r^{n-1}}\int\limits_{S(y_0,
r)}Q(y)\,d\mathcal{H}^{n-1}(y)\,, \end{equation*}
where $\omega_{n-1}$ denotes the area of the unit sphere ${\Bbb
S}^{n-1}$ in ${\Bbb R}^n.$

\medskip
We say that a function ${\varphi}:D\rightarrow{\Bbb R}$ has a {\it
finite mean oscillation} at a point $x_0\in D,$ write $\varphi\in
FMO(x_0),$ if
$$\limsup\limits_{\varepsilon\rightarrow
0}\frac{1}{\Omega_n\varepsilon^n}\int\limits_{B( x_0,\,\varepsilon)}
|{\varphi}(x)-\overline{{\varphi}}_{\varepsilon}|\ dm(x)<\infty\,,
$$
where $\overline{{\varphi}}_{\varepsilon}=\frac{1}
{\Omega_n\varepsilon^n}\int\limits_{B(x_0,\,\varepsilon)}
{\varphi}(x) \,dm(x)$ and $\Omega_n$ denotes the volume of the unit
ball ${\Bbb B}^n$ in ${\Bbb R}^n.$
We also say that a function ${\varphi}:D\rightarrow{\Bbb R}$ has a
finite mean oscillation at $A\subset \overline{D},$ write
${\varphi}\in FMO(A),$ if ${\varphi}$ has a finite mean oscillation
at any point $x_0\in A.$ Let $h$ be a chordal metric in
$\overline{{\Bbb R}^n},$
\begin{equation}\label{eq3C}
h(x,\infty)=\frac{1}{\sqrt{1+{|x|}^2}}\,,\quad
h(x,y)=\frac{|x-y|}{\sqrt{1+{|x|}^2} \sqrt{1+{|y|}^2}}\qquad x\ne
\infty\ne y\,.
\end{equation}
and let $h(E):=\sup\limits_{x,y\in E}\,h(x,y)$ be a chordal diameter
of a set~$E\subset \overline{{\Bbb R}^n}$ (see, e.g.,
\cite[Definition~12.1]{Va}).

Following~\cite[Section~2.4]{NP}, we say that a domain $D\subset
{\Bbb R}^n,$ $n\geqslant 2,$ is {\it uniform with respect to
$p$-modulus}, if for any $r>0$ there is $\delta>0$ such that the
inequality
$$
M_p(\Gamma(F^{\,*},F, D))\geqslant \delta
$$
holds for any continua $F, F^*\subset D$ with $h(F)\geqslant r$ and
$h(F^{\,*})\geqslant r,$ where $h$ is a chordal metric defined
in~(\ref{eq3C}). When $p=n,$ the prefix ``relative to $p $-modulus''
is omitted. Note that, this definition slightly different from the
``classical'' given in \cite[Chapter~2.4]{NP}, where the sets $F$
and $F^*\subset D $ are assumed to be arbitrary connected.

Given $n-1<p\leqslant n,$ $a, b\in D,$ $a\ne b,$ a Lebesgue
measurable function $Q:D\rightarrow [0, \infty]$ and $\delta>0$ we
define the family $\frak{F}^{p, \delta}_{Q, a, b}(D)$ of all
mappings $f:D\rightarrow{\Bbb R}^n$ which satisfy the
relations~(\ref{eq2*A})--(\ref{eqA2}) at every point $y_0\in {\Bbb
R}^n$ and for every $0<r_1<r_2<d_0=\sup\limits_{y\in f(D)}|y-y_0|$
such that $h(f(a), f(b))\geqslant \delta.$ The following statement
holds, cf. \cite[Lemma~1]{ST$_2$}.

\medskip
\begin{theorem}\label{th1} {\it Let $D$ be a domain in ${\Bbb R}^n,$
$n\geqslant 2,$ let $n-1<p\leqslant n,$ let $a, b\in D,$ $a\ne b,$
let $\delta>0$ and let $Q:D\rightarrow [0, \infty]$ be a Lebesgue
measurable function. Assume that the following conditions hold:

\medskip
1) the domain $D$ is $p$-uniform,

2) the family $\frak{F}^{p, \delta}_{Q, a, b}(D)$ is equicontinuous
at the points $a$ and $b,$

\medskip
3) the function $Q$ satisfies at least one of the following
conditions:

\medskip
$3_1)$ $Q\in FMO(\overline{{\Bbb R}^n});$

\medskip
$3_2)$ for any $y_0\in \overline{{\Bbb R}^n}$ there is
$\delta(y_0)>0$ such that
\begin{equation}\label{eq5D}
\int\limits_{\varepsilon}^{\delta(y_0)}
\frac{dt}{t^{\frac{n-1}{p-1}}q_{y_0}^{\frac{1}{p-1}}(t)}<\infty\,,\quad
0<\varepsilon<\varepsilon_0, \quad\int\limits_{0}^{\delta(y_0)}
\frac{dt}{t^{\frac{n-1}{p-1}}q_{y_0}^{\frac{1}{p-1}}(t)}=\infty\,.
\end{equation}

Then the following holds: given $\varepsilon>0$ there is
$\delta_1(\varepsilon)>0$ such that $h(f(C))\geqslant \delta_1$ for
any $f\in \frak{F}^{p, \delta}_{Q, a, b}(D)$ and any continuum
$C\subset D$ with $h(C)\geqslant \varepsilon.$  }
\end{theorem}

\medskip
Let us emphasize that condition~(\ref{eq5D}) is exact, as indicated
by the following statement.

\medskip
\begin{theorem}\label{th3} {\it Given $n-1<p\leqslant n$
and a locally integrable function $Q:D\rightarrow
[0, \infty]$ for which
\begin{equation*}\label{eq5E}
\int\limits_{0}^{\delta(y_0)}
\frac{dt}{t^{\frac{n-1}{p-1}}q_{y_0}^{\frac{1}{p-1}}(t)}<\infty
\end{equation*}
at least at one point $y_0\in D$ and some $\delta(y_0)>0$ there are
$\varepsilon_0>0,$ $a, b\in D,$ $\delta>0,$ a sequence
$f_k:D\rightarrow {\Bbb R}^n,$ $k=1,2,\ldots ,$ of mappings
satisfying the relations (\ref{eq2*A})--(\ref{eqA2}) and a sequence
of continua $C_k$ in $D,$ $k=1,2,\ldots$ such that $h(f_k(a),
f_k(b))\geqslant \delta>0$ and $h(C_k)\geqslant \varepsilon_0$ for
all $k\in {\Bbb N},$ however, $h(f_k(C_k))<\frac{1}{k},$
$k=1,2,\ldots .$ This sequence $C_k$ be be chosen belonging to the
fixed compactum $K$ in~$D.$ The points $a, b $ also may be chosen
such that the family $f_k(x)$ is equicontinuous at $a$ and $b.$ The
domain $D$ may be $p$-uniform, or not. }
\end{theorem}

We should note that, if continua $C$ from Theorem~\ref{th1} is
itself contained in a fixed compactum $K$ inside $D,$ then no
conditions on the geometry of the boundary of the domain $D$ are
required at all. This assertion is contained in the following
theorem.

\medskip
\begin{theorem}\label{th2} {\it Let $D$ be a domain in ${\Bbb R}^n,$
$n\geqslant 2,$ let $n-1<p\leqslant n,$ let $a, b\in D,$ $a\ne b,$
let $\delta>0$ and let $Q:D\rightarrow [0, \infty]$ be a Lebesgue
measurable function. Let $K$ be a non-degenerate compactum inside
$D.$ Assume that the following conditions hold:

\medskip
1) the family $\frak{F}^{p, \delta}_{Q, a, b}(D)$ is equicontinuous
at the points $a$ and $b,$

\medskip
2) the function $Q$ satisfies at least one of the following
conditions:

\medskip
$2_1)$ $Q\in FMO(\overline{{\Bbb R}^n});$

\medskip
$2_2)$ for any $y_0\in \overline{{\Bbb R}^n}$ there is
$\delta(y_0)>0$ such that
\begin{equation*}\label{eq5D_A}
\int\limits_{\varepsilon}^{\delta(y_0)}
\frac{dt}{t^{\frac{n-1}{p-1}}q_{y_0}^{\frac{1}{p-1}}(t)}<\infty\,,\quad
0<\varepsilon<\varepsilon_0, \quad\int\limits_{0}^{\delta(y_0)}
\frac{dt}{t^{\frac{n-1}{p-1}}q_{y_0}^{\frac{1}{p-1}}(t)}=\infty\,.
\end{equation*}

Then the following holds: given $\varepsilon>0$ there is
$\delta_1(\varepsilon)>0$ such that $h(f(C))\geqslant \delta_1$ for
any $f\in \frak{F}^{p, \delta}_{Q, a, b}(D)$ and any continuum
$C\subset K$ with $h(C)\geqslant \varepsilon.$  }
\end{theorem}

\medskip
\begin{remark}
In Theorems~\ref{th1} and~\ref{th2}, the mappings are neither open
nor discrete. The class of mappings $\frak{F}^{p, \delta}_{Q, a,
b}(D)$ involved in these theorems is not assumed to be a normal
family of mappings, although it is assumed to be equicontinuous at
only two given points $a$ and $b.$ There are also no conditions
implying the possibility of a continuous extension of these mappings
to the boundary of the domain $D.$ This remark may be important in
the context of some other statements on this topic; see, for
example, \cite{Cr$_1$}, \cite{Cr$_2$}, \cite{SevSkv} and
\cite{ST$_2$}.
\end{remark}

\section{The main Lemma}

\begin{lemma}\label{lem1}
{\it\, The statement of Theorem~\ref{th1} is true if under the
conditions of this theorem condition~3) is replaced by the following
condition: for any $y_0\in\overline{{\Bbb R}^n}$ there is a Lebesgue
measurable function $\psi:(0, \varepsilon_0)\rightarrow (0,\infty)$
such that
\begin{equation*}\label{eq3.7.1}
0<I(\varepsilon,
\varepsilon_0):=\int\limits_{\varepsilon}^{\varepsilon_0}\psi(t)\,dt
< \infty\,,
\end{equation*}
while there exists a function $\alpha=\alpha(\varepsilon,
\varepsilon_0)\geqslant 0$ such that
\begin{equation} \label{eq3.7.2}
\int\limits_{A(y_0, \varepsilon, \varepsilon_0)}
Q(y)\cdot\psi^{\,p}(|y-y_0|)\,dm(y)= \alpha(\varepsilon,
\varepsilon_0)\cdot I^p(\varepsilon, \varepsilon_0)\,,\end{equation}
where $A(y_0, \varepsilon, \varepsilon_0)$ is defined in
(\ref{eq1**}).}
\end{lemma}

\medskip
Here the conditions mentioned above for $y_0=\infty$ must be
understood as conditions for the function
$\widetilde{Q}(y):=Q(y/|y|^2)$ at the origin and for the
corresponding mapping $\widetilde{f}:=\psi_1\circ f$ instead of~$f,$
where $\psi_1(y)=\frac{y}{|y|^2}.$

\medskip
\begin{proof}
Let us prove the statement of the lemma by contradiction, i.e.,
assume that there exists $\varepsilon_1>0$ such that for every $k\in
{\Bbb N}$ there is a continuum $C_k\subset D$ and a mapping $f_k\in
\frak{F}^{p, \delta}_{Q, a, b}(D)$ such that $h(C_k)\geqslant
\varepsilon_1,$ however, $h(f_k(C_k))<1/k.$ Let us join the points
$a$ and $b$ with a path $\gamma:[0, 1]\rightarrow D,$ $\gamma(0)=a,$
$\gamma(1)=b,$ in~$D.$ It follows from the conditions of the lemma
that $h(f_k(\gamma))\geqslant \delta$ for any $k=1,2,\ldots .$ Due
to the compactness of $\overline{{\Bbb R}^n},$ we may assume that
$y_k\rightarrow y_0\in\overline{{\Bbb R}^n}$ as $k\rightarrow\infty$
for some sequence $y_k\in f_k(C_k)$ and some $y_0\in \overline{{\Bbb
R}^n}.$ Let us firstly consider the case when $y_0\ne\infty.$ Since
$h(f_k(C_k))<1/k$ for any $k=1,2,\ldots ,$ there exists $k_0\in
{\Bbb N}$ such that $f_k(C_k)\subset B(y_0, \delta/3).$ Now,
$d(f_k(\gamma))\geqslant h(f_k(\gamma))\geqslant \delta,$
$k=1,2,\ldots ,$ and, moreover, the points $f_k(a)$ and $f_k(b)$ may
not belong to $B(y_0, \delta)$ simultaneously. Let us to consider
two possibilities:

\medskip
1) Exactly one of the points $f_k(a)$ or $f_k(b)$ belongs to the
ball $B_h(y_0, \delta/3),$ where $B_h(y_0,
\delta/3)=\{z\in\overline{{\Bbb R}^n}: h(z, y_0)<\delta/3\}.$ We may
consider that $f_k(a)\in B_h(y_0, \delta/3)$ and $f_k(b)\not\in
B_h(y_0, \delta/3).$ By \cite[Theorem~1.I.5, \S46]{Ku}
$|f_k(\gamma)|\cap  S_h(y_0, \delta/3)\ne\varnothing,$ where
$S_h(y_0, \delta/3)=\{z\in\overline{{\Bbb R}^n}: h(z,
y_0)=\delta/3\}.$ Let $t_k=\sup\limits_{t\in [0, 1]:
f_k(\gamma(t))\in S_h(y_0, \delta/3)}t.$ Now, the path
$f_k(\gamma)|_{[t_k, 1]}$ belongs to ${\Bbb R}^n\setminus B_h(y_0,
\delta/3).$ By the triangle inequality,
\begin{gather}
\delta\leqslant h(f_k(\gamma(0)),f_k(\gamma(1))) \leqslant
h(f_k(\gamma(0)), f_k(\gamma(t_k)))\nonumber\\
\label{eq2A} +h(f_k(\gamma(t_k)), f_k(\gamma(1)))\leqslant
\frac{\delta}{3}+h(f_k(\gamma(t_k)), f_k(\gamma(1)))\,.
\end{gather}
It follows from~(\ref{eq2A}) that
\begin{equation}\label{eq1}
h(f_k(\gamma(t_k)),f_k(\gamma(1)))\geqslant \frac{2\delta}{3}
\end{equation}
for all respective $k\in {\Bbb N}.$

We set $D_k:=\bigl|f_k(\gamma)|_{[t_k, 1]}\bigr|.$ Observe that
$D_k\subset {\Bbb R}^n\setminus B_h(y_0, \delta/3)$ by the
construction. We also set $D^{\,*}_k=\bigl|\gamma|_{[t_k,
1]}\bigr|.$ Observe that, there exists $\delta_*>0$ such that
$h(D^{\,*}_k)\geqslant \delta_*$ for every $k\in {\Bbb N}.$
Otherwise, $h(D^{\,*}_{k_l})\rightarrow 0$ as $l\rightarrow \infty,$
but this contradicts with the equicontinuity of the family $f_k$ at
the point $b,$ because $\gamma(t_{k_l})\rightarrow b$ as
$l\rightarrow\infty$ and simultaneously $h(f_{k_l}(\gamma(t_{k_l}),
f_{k_l}(1)))=h(D_{k_l})\geqslant \frac{2\delta}{3}$ for infinitely
many $l$ by~(\ref{eq1}).

\medskip
2) Exactly two points $f_k(a)$ or $f_k(b)$ lye outside the ball
$B_h(y_0, \delta/3).$ If $\bigl|f_k(\gamma)\bigr|\subset {\Bbb
R}^n\setminus B_h(y_0, \delta/3),$ we set
$D_k:=\bigl|f_k(\gamma)\bigr|$ and $D^{\,*}_k=\bigl|\gamma\bigr|.$
Now, $h(D^{\,*}_k)\geqslant \delta_*$ for every $k\in {\Bbb N}$ and
some $\delta_*>0,$ moreover, $D_k\subset {\Bbb R}^n\setminus
B_h(y_0, \delta/3).$

\medskip
Otherwise, if $|f_k(\gamma)|\cap ({\Bbb R}^n\setminus B_h(y_0,
\delta/3))\ne\varnothing \ne |f_k(\gamma)|\cap B_h(y_0, \delta/3)),$
we set $$t^{\,*}_k=\inf\limits_{t\in [0, 1]: f_k(\gamma(t))\in {\Bbb
R}^n\setminus B_h(y_0, \delta/3)}t\,,\qquad t_k=\sup\limits_{t\in
[0, 1]: f_k(\gamma(t))\in B_h(y_0, \delta/3)}t.$$ Now, the paths
$\widetilde{\gamma_k^{(1)}}=f_k(\gamma)|_{[0, t^{\,*}_k]}$ and
$\widetilde{\gamma_k^{(2)}}=f_k(\gamma)|_{[t_k, 1]}$ belongs to
${\Bbb R}^n\setminus B_h(y_0, \delta/3).$ Arguing similarly
to~(\ref{eq2A}) we conclude that at least one of the paths
$\widetilde{\gamma_k^{(1)}}$ or $\widetilde{\gamma_k^{(2)}}$ has a
chordal diameter not less than $\delta/3.$ Let $D_k$ be locus of
this path, and let $D^{\,*}_k$ be the locus of its corresponding
``pre-image'' in $D.$ Arguing similarly to the above case~1) we
obtain that $D_k\subset {\Bbb R}^n\setminus B_h(y_0, \delta/3)$ and
$h(D^{\,*}_k)\geqslant \delta_*$ for every $k\in {\Bbb N}$ and some
$\delta_*>0,$ as well.

\medskip
Thus, in both cases 1) or 2) we have that $D_k\subset {\Bbb
R}^n\setminus B_h(y_0, \delta/3)$ and $h(D^{\,*}_k)\geqslant
\delta_*$ for every $k\in {\Bbb N}$ and some $\delta_*>0.$ Setting
$\Gamma_k=\Gamma (C_k, D^{\,*}_k, D)$ by the $p$-uniformity of the
domain $D$ we obtain that
\begin{equation}\label{eq1A}
M_p(\Gamma_k)\geqslant \delta_1>0
\end{equation}
for any $k=1,2,\ldots $ and some $\delta_1>0.$ Let us consider
$\varepsilon_0$ from the conditions of the lemma. Reducing it, if
necessary, we may consider that $B(y_0, \varepsilon_0)\subset
B_h(y_0, \delta/3).$

Since $h(f_k(C_k))<1/k,$ we may assume that
\begin{equation*}\label{eq3B}
f_{k}(C_{k})\subset B(y_0, 1/k)\,, k=1,2,\ldots\,.
\end{equation*}
Let $k_0\in {\Bbb N}$ be such that $B(y_0, 1/k)\subset B(y_0,
\varepsilon_0)$ for all $k\geqslant k_0.$
In this case, observe that
\begin{equation}\label{eq3G}
f_k(\Gamma_k)>\Gamma(S(y_0, 1/k), S(y_0, \varepsilon_0), A(y_0,
1/k,\varepsilon_0))\,.
\end{equation}
Indeed, let $\widetilde{\gamma}\in f_k(\Gamma_k).$ Then
$\widetilde{\gamma}(t)=f_{k}(\Delta(t)),$ where $\Delta\in
\Gamma_k,$ $\Delta:[0, 1]\rightarrow D,$ $\Delta(0)\in C_k,$
$\gamma(1)\in D^{\,*}_k.$ Since $f_k(D^{\,*}_k)=D_k\subset {\Bbb
R}^n\setminus B_h(y_0, \delta/3)$ and $B(y_0, \varepsilon_0)\subset
B_h(y_0, \delta/3)$ by the construction,$|\widetilde{\gamma}|\cap
B(y_0, \varepsilon_0)\ne\varnothing \ne |\widetilde{\gamma}|\cap
({\Bbb R}^n\setminus B(y_0, \varepsilon_0)).$ Then by
\cite[Theorem~1.I.5, \S46]{Ku} there exists $0<t_1<1$ such that
$f_k(\Delta(t_1))\in S(y_0, 1/k).$ Setting $\Delta_1:=\Delta|_{[t_1,
1]},$ we may assume that $f_k(\Delta(t))\not\in B(y_0, 1/k)$ for
every $t\geqslant t_1.$ Arguing similarly, we obtain a point $t_2\in
(t_1, 1]$ such that $f_k(\Delta(t_2))\in S(y_0, \varepsilon_0).$
Setting $\Delta_2:=\Delta|_{[t_1, t_2]},$ we may assume that
$f_k(\Delta(t))\in B(y_0, \varepsilon_0)$ for all $t\in [t_1, t_2].$
Then the path $f_k(\Delta_2)$ is a subpath of
$f_k(\Delta)=\widetilde{\gamma}$ which belongs to $\Gamma(S(y_0,
1/k), S(y_0, \varepsilon_0), A(y_0, 1/k,\varepsilon_0)).$ The
relation~(\ref{eq3G}) is established.

\medskip
It follows from~(\ref{eq3G}) that
\begin{equation}\label{eq3H}
\Gamma_k>\Gamma_{f_k}(y_0, 1/k, \varepsilon_0)\,.
\end{equation}
We set $$\eta_{k}(t)=\left\{
\begin{array}{rr}
\psi(t)/I(1/k, \varepsilon_0), & t\in (1/k, \varepsilon_0)\,,\\
0,  &  t\not\in (1/k, \varepsilon_0)\,,
\end{array}
\right. $$
where $I(1/k, \varepsilon_0)=\int\limits_{1/k}^{\varepsilon_0}\,\psi
(t)\, dt.$ Observe that
$\int\limits_{1/k}^{\varepsilon_0}\eta_{k}(t)\,dt=1.$ Then by the
relations~(\ref{eq3.7.2}) and~(\ref{eq3H}), due to the definition of
$f_k$ in~(\ref{eq2*A}) we obtain that
\begin{gather*}
M_p(\Gamma_k)\leqslant M_p(\Gamma_{f_k}(y_0, 1/k, \varepsilon_0))\\
\leqslant \frac{1}{I^p(1/k, \varepsilon_0)}\int\limits_{A(y_0, 1/k,
\varepsilon_0)} Q(y)\cdot\psi^{\,p}(|y-y_0|)\,dm(y)\rightarrow
0\quad \text{as}\quad k\rightarrow\infty\,. \end{gather*}
The latter contradicts with~(\ref{eq1A}). Lemma is proved for the
case $y_0\ne\infty.$

Let us consider the case $y_0=\infty.$ Applying the inversion
$\psi(y)=\frac{y}{|y|^2},$ we consider the family of mappings
$\widetilde{f}_k:=\psi\circ f_k.$ Now, due to the compactness of
$\overline{{\Bbb R}^n},$ we may assume that $y_k\rightarrow
0\in\overline{{\Bbb R}^n}$ as $k\rightarrow\infty$ for some sequence
$y_k\in \widetilde{f}_k(C_k)$ while $\widetilde{f}_k$ satisfy the
relation~(\ref{eq2*A}) for $y_0=0$ with
$\widetilde{Q}(y)=Q\left(\frac{y}{|y|^2}\right).$ Then the proof of
the lemma, starting from the second paragraph for mappings
$\widetilde{f}_k,$ $k=1,2,\ldots ,$ is completely similar to the
case $y_0\ne\infty.$~$\Box$
\end{proof}

The following statement may be found in~\cite[Lemmas~1.3 and
1.4]{Sev$_2$}.

\medskip
\begin{proposition}\label{pr6}
{\it\, Let $Q:{\Bbb R}^n\rightarrow [0,\infty],$ $n\geqslant 2,$
$n-1<p\leqslant n,$ be a Lebesgue measurable function and let
$x_0\in {\Bbb R}^n.$ Assume that either of the following conditions
holds

\noindent (a) $Q\in FMO(x_0),$

\noindent (b)
$q_{x_0}(r)\,=\,O\left(\left[\log{\frac1r}\right]^{n-1}\right)$ as
$r\rightarrow 0,$

\noindent (c) for some small $\delta_0=\delta_0(x_0)>0$ we have the
relations
$$
\int\limits_{\delta}^{\delta_0}\frac{dt}{t^{\frac{n-1}{p-1}}q_{x_0}^{\frac{1}{p-1}}(t)}<\infty,\qquad
0<\delta<\delta_0,
$$
and
$$
\int\limits_{0}^{\delta_0}\frac{dt}{t^{\frac{n-1}{p-1}}q_{x_0}^{\frac{1}{p-1}}(t)}=\infty\,.
$$
Then  there exist a number $\varepsilon_0\in(0,1)$ and a function
$\psi:(0, \varepsilon_0)\rightarrow [0, \infty)$ such that the
relation
$$
\int\limits_{\varepsilon<|x-x_0|<\varepsilon_0}Q(x)\cdot\psi^p(|x-x_0|)
\ dm(x)=o(I^p(\varepsilon, \varepsilon_0))\,,
$$
holds as $\varepsilon\rightarrow 0,$
where
$$
0<I(\varepsilon, \varepsilon_0)
=\int\limits_{\varepsilon}^{\varepsilon_0}\psi(t)\,dt < \infty
\qquad \forall\quad\varepsilon \in(0, \varepsilon_1)
$$
for some $0<\varepsilon_1<\varepsilon_0.$ }
\end{proposition}

\medskip
{\it Proof of Theorem~\ref{th1}} immediately follows by
Lemma~\ref{lem1} and Proposition~\ref{pr6}.~$\Box$

\section{Analogs of N\"{a}kki theorems for $p$-modulus}

Let $E_0,$ $E_1$ be sets in $D\subset {\Bbb R}^n$. The following
estimate holds (see \cite[Theorem~4]{Car}).

\begin{proposition}\label{pr4}
{\it Let $A(0, a, b)=\{a<|x|<b\}$ be a ring containing in $D\subset
{\Bbb R}^n$ such that $S(0, r)$ intersects $E_0$ and $E_1$ for any
$r\in (a,b)$ where $E_0\cap E_1=\varnothing.$ Then for any
$p\in(n-1, n)$
\begin{equation*}
M_p(\Gamma(E_0, E_1, D))\geqslant
\frac{2^nb_{n,p}}{n-p}(b^{n-p}-a^{n-p})\,,
\end{equation*}
where $b_{n,p}$ is a constant depending only $n$ and $p$.}
\end{proposition}

\medskip
The following result is true, see \cite[Theorem~1.1]{SKNI},
cf.~\cite[Lemma~1.15]{Na$_1$}.

\begin{proposition}\label{pr1}
{\it Let $D$ be a domain in ${\Bbb R}^n,$ $n\geqslant 2,$ and let
$p>n-1.$ If $A$ and $A^{\,*}$ are (nondegenerate) continua in $D,$
then $M_p(\Gamma(A, A^{\,*}, D))>0.$}
\end{proposition}

\medskip
The following statement is proved in \cite[Theorem~3.1]{Na$_2$} for
$p=n$.

\medskip
\begin{proposition}\label{pr2}
{\it\, Let $n-1<p<n,$ $F_1,$ $F_2,$ $F_3$ be three sets in a domain
$D$ and let $\Gamma_{ij}=\Gamma(F_1, F_2, D),$ $1\leqslant i,
j\leqslant 3.$ Then
$$M_p(\Gamma_{12})\geqslant 3^{-p}\min\{M_p(\Gamma_{13}),
M_p(\Gamma_{23}),\inf M_p(\Gamma(|\gamma_{13}|, |\gamma_{23}|,
D))\}$$
where the infimum is taken over all rectifiable paths
$\gamma_{13}\in \Gamma_{13},$ $\gamma_{23}\in \Gamma_{23}.$}
\end{proposition}

\medskip
\begin{proof} We may assume that $F_1,$ $F_2,$ $F_3$ are nonempty sets, for
otherwise there is nothing to prove. Choose $\rho\in {\rm
adm}\,\Gamma_{12}.$ If at least one of the conditions hold
\begin{equation}\label{eq28***}
\int\limits_{\gamma_{1, 3}}\rho \,|dx|\geqslant 1/3\,, \qquad 
\int\limits_{\gamma_{2, 3}}\rho\, |dx|\geqslant 1/3\,,
\end{equation}
for $\gamma_{1, 3}\in \Gamma_{1, 3},$ $\gamma_{2, 3}\in \Gamma_{2,
3},$ then we obtain that $3\rho\in {\rm adm\,}\Gamma_{1, 3}$ or
$3\rho\in {\rm adm\,}\Gamma_{2, 3}.$ Now, we obtain that
\begin{equation}\label{eq31***}
\int\limits_{D} \rho^p(x)\,dm(x)\geqslant
3^{\,-p}\min\{M_p(\Gamma_{1, 3}), M_p(\Gamma_{2, 3})\}\,.
\end{equation}
If neither relation in~(\ref{eq28***}) is true for some rectifiable
paths $\gamma_{1, 3}\in \Gamma_{1, 3},$ $\gamma_{2, 3}\in \Gamma_{2,
3},$ Then
\begin{equation}\label{eq1C}
\int\limits_{\alpha}\rho \,|dx|\geqslant 1/3
\end{equation}
for every rectifiable path $\alpha\in\Gamma(|\gamma_{13}|,
|\gamma_{23}|, D).$ Thus $3\rho\in{\rm adm}\,\Gamma(|\gamma_{13}|,
|\gamma_{23}|, D)$ which implies
\begin{equation}\label{eq1B}
M_p(\Gamma(|\gamma_{13}|, |\gamma_{23}|, D))\geqslant\int\limits_{D}
\rho^p(x)\,dm(x)\geqslant 3^{\,-p}M_p(\Gamma(|\gamma_{13}|,
|\gamma_{23}|, D))\,.
\end{equation}
Since $\rho\in {\rm adm}\,\Gamma_{12}$ was arbitrary and since
either~(\ref{eq31***}) or (\ref{eq1B}) must be true, the assertion
follows.
\end{proof}

\medskip
The following statement is proved in \cite[Theorem~3.3]{Na$_2$} for
$p=n$.

\medskip
\begin{proposition}\label{pr3}
{\it\, Let $n-1<p<n,$ Let $F_1,$ $F_2,$ $F_3$ be three sets in a
domain $D,$ let $D$ contain the spherical ring $A(y_0, r_1, r_2),$
$0<r_1<r_2<\infty,$ let $F_3$ lie in $B(y_0, r_1)$ and let
$\Gamma_{ij}$ be as in Proposition~\ref{pr2}. If one of the three
conditions

(1) $F_i$ lies in ${\Bbb R}^n\setminus B(y_0, r_2),$ $i=1, 2,$

(2) $F_1$ lies in ${\Bbb R}^n\setminus B(y_0, r_2)$ and $F_2$ is
connected with $d(F_2)\geqslant 2r_2,$

(3) $F_i$ is connected with $d(F_i)\geqslant 2r_2,$ $i=1,2,$ is
satisfied, then

$$M_p(\Gamma_{12})\geqslant 3^{\,-p}\min\{M_p(\Gamma_{13}),
M_p(\Gamma_{23}),
\frac{2^nb_{n,p}}{n-p}((r_2)^{n-p}-{r_1}^{n-p})\}\,,$$
where $b_{n,p}$ is a positive constant depending only on $n$ and
$p.$}
\end{proposition}

\medskip
\begin{proof}
We may assume that $F_1,$ $F_2,$ $F_3$ are nonempty sets. If (1) is
satisfied, then the assertion follows directly from
Propositions~\ref{pr2} and \ref{pr4}. Assume next that (2) or (3) is
satisfied. Choose $\rho\in {\rm adm}\,\Gamma_{1, 2}.$ If at least
one of the conditions in~(\ref{eq28***}) holds for every rectifiable
path $\gamma_{1, 3}\in \Gamma_{1, 3}$ and $\gamma_{2, 3}\in
\Gamma_{2, 3},$ then (\ref{eq31***}) holds. If neither the first and
the second relations in~(\ref{eq28***}) hold for some rectifiable
paths $\gamma_{1, 3}\in \Gamma_{1, 3}$ and $\gamma_{2, 3}\in
\Gamma_{2, 3},$  then (\ref{eq1C}) holds for every rectifiable path
$\alpha\in \Gamma(F_1\cup|\gamma_{1, 3}, F_2\cup|\gamma_{2, 3}, D).$
Therefore, since $S(y_0, t)$ meets both $F_1\cup|\gamma_{1, 3}$ and
$F_2\cup|\gamma_{2, 3}$ for $r_1<t<r_2$ and since $D$ contains the
spherical ring $A(y_0, r_1, r_2),$ we obtain
\begin{equation}\label{eq1D}
\int\limits_{D} \rho^p(x)\,dm(x)\geqslant
\frac{2^nb_{n,p}}{n-p}(r_2^{n-p}-r_1^{n-p})\,.
\end{equation}
Finally, since $\rho\in {\rm adm}\,\Gamma_{1, 2}$ was arbitrary and
since either (\ref{eq31***}) or (\ref{eq1D}) must be true, the
assertion follows.
\end{proof}

\medskip
The following statement holds.

\medskip
\begin{theorem}\label{th4}
{\it Let $n-1<p\leqslant n,$ and let $\frak{F}$ be a collection of
connected sets in a domain $D$ and let $\inf h(F)>0,$ $F\in
\frak{F}.$ Then $\inf\limits_{F\in \frak{F}} M_p(\Gamma(A, F, D))>0$
either for each or for no continuum $A$ in~$D.$}
\end{theorem}

\medskip
\begin{proof}
The proof in the case $p=n$ was established by N\"{a}kki, see
\cite[Theorem~4.1]{Na$_2$}. Our task is to establish this fact for
an arbitrary order of the modulus $p\in(n-1, n).$ Let $A$ and
$A^{\,*}$ be two continua in $D$ and let $M_p(\Gamma(A, F,
D))\geqslant\delta>0$ for all $F\in \frak{F}.$ Assume first that
$A\cap A^{\,*}=\varnothing.$ Choose a number $r>0$ so that
$\overline{B(a, 2r)}\subset D$ for each point $a\in A$ and so that
$0<4r<\min\{\inf h(F), d(A, A^{\,*})\}.$ Let $A_1,\ldots, A_q$ be a
finite covering of $A$ by closed balls with centers $a_i\in A,$
$i=1,2,\ldots, q,$ and radii $r.$ Let $M_p(\Gamma(A_i, A^{\,*},
D))=\delta_i.$ By Proposition~\ref{pr1} $\delta_i>0.$ We claim that
$$M_p(\Gamma(A^{\,*}, F, D))\geqslant 3^{\,-p}\min\bigl\{\delta/q, \delta_1,\ldots, \delta_q,
\frac{2^nb_{n,p}}{n-p}\bigl((2r)^{n-p}-(r)^{n-p}\bigr)\bigr\}\,.$$
For this, let $F\in \frak{F}.$ Then by the subadditivity of the
modulus, $0<\delta\leqslant M_p(\Gamma(A, F,
D))\leqslant\sum\limits_{i=1}^qM_p(\Gamma(A_i, F, D)),$ so that
$M_p(\Gamma(A_i, F, D))\geqslant \delta/q$ for some $i.$ Fix this
$i.$ Since $A^{\,*}\cap B(a_i, 2r)=\varnothing$ and since $d(F)>
4r,$ the assertion follows from Proposition~\ref{pr3}(2) by setting
$F_1=A^{\,*},$ $F_2=F,$ and $F_3=A_i.$ In the preceding argument we
assumed that $A\cap A^{\,*}=\varnothing.$ Suppose now that $A\cap
A^{\,*}\ne\varnothing.$ If the set $D\setminus (A \cup A^{\,*})$ is
nonempty, and therefore contains a continuum $A^{\,\prime},$ we may
apply the above procedure first to the sets $A, A^{\,\prime}$ and
then to $A^{\,\prime}, A^{\,*}.$ This completes the proof of
Theorem~\ref{th4}.~$\Box$
\end{proof}

\section{Proof of Theorem~\ref{th2}}

Following \cite[section~7.22]{He}, given a real-valued function $u$
in a metric space $X,$ a Borel function $\rho\colon X\rightarrow [0,
\infty]$ is said to be an {\it upper gradient} of a function
$u:X\to{\Bbb R}$ if $|u(x)-u(y)|\leqslant
\int\limits_{\gamma}\rho\,|dx|$ for each rectifiable curve $\gamma$
joining $x$ and $y$ in $X.$ Let $(X, \mu)$ be a metric measure space
and let $1\leqslant p<\infty.$ We say that $X$ admits {\it $(1;
p)$-Poincare inequality} if there is a constant $C\geqslant 1$ such
that
$$\frac{1}{\mu(B)}\int\limits_{B}|u-u_B|d\mu(x)\leqslant C\cdot({\rm
diam\,}B)\left(\frac{1}{\mu(B)} \int\limits_{B}\rho^n
d\mu(x)\right)^{1/n}$$
for all balls $B$ in $X,$ for all bounded continuous functions $u$
on $B,$ and for all upper gradients $\rho$ of $u.$ Metric measure
spaces where the inequalities
$$\frac{1}{C}R^{n}\leqslant \mu(B(x_0,
R))\leqslant CR^{n}$$
hold for a constant $C\geqslant 1$, every $x_0\in X$  and all
$R<{\rm diam}\,X$, are called {\it Ahlfors $n$-regular.} The
following result holds (see \cite[Proposition~4.7]{AS}).

\begin{proposition}\label{pr_2}
Let $X$ be a $Q$-Ahlfors regular metric measure space that supports
$(1; p)$-Poincar\'{e} inequality for some $p>1$ such that
$Q-1<p\leqslant Q.$ Then there exists a constant $C>0$ having the
property that, for $x\in X,$ $R>0$ and continua $E$ and $F$ in $B(x,
R),$
$$M_p(\Gamma(E, F, X))\geqslant \frac{1}{C}
\cdot\frac{\min\{{\rm diam}\,E, {\rm diam}\,F\}}{R^{1+p-Q}}\,.$$
\end{proposition}
Let $(X, \mu)$ be a metric space with measure $\mu.$ For each real
number $n\ge 1,$ we define {\it the Loewner function} $\phi_n:(0,
\infty)\rightarrow [0, \infty)$ on $X$ as
$$\phi_n(t)=\inf\{M_n(\Gamma(E, F, X)): \Delta(E, F)\leqslant t\}\,,$$
where the infimum is taken over all disjoint nondegenerate continua
$E$ and $F$ in $X$ and
$$\Delta(E, F):=\frac{{\rm dist}\,(E,
F)}{\min\{{\rm diam\,}E, {\rm diam\,}F\}}\,.$$
A pathwise connected metric measure space $(X, \mu)$ is said to be a
{\it Loewner space} of exponent $n,$ or an $n$-Loewner space, if the
Loewner function $\phi_n(t)$ is positive for all $t> 0$ (see
\cite[section~2.5]{MRSY} or \cite[Ch.~8]{He}). Observe that ${\Bbb
R}^n$ and ${\Bbb B}^n\subset {\Bbb R}^n$ are Loewner spaces (see
\cite[Theorem~8.2 and Example~8.24(a)]{He}). As known, a condition
$\mu(B(x_0, r))\geqslant C\cdot r^n$ holds in Loewner spaces $X$ for
a constant $C>0$, every point $x_0\in X$  and all $r<{\rm diam}\,X.$

\medskip
\begin{proposition}\label{pr1X}
An open ball is an Ahlfors $n$-regular metric space in which $(1;
p)$-Poincare inequality holds for every $p\geqslant 1.$
\end{proposition}

\begin{proof}
By comments given above, the ball is Ahlfors $n$-regular. By
\cite[Theorem~10.5]{HK}, the $(1; p)$-Poincare inequality holds in
this ball for any $p\geqslant 1.$ ~$\Box$
\end{proof}

\medskip
\begin{lemma}\label{lem4}
{\it\, The statement of Theorem~\ref{th2} is true if in the
conditions of this theorem condition~2) is replaced by the following
condition: for any $y_0\in\overline{{\Bbb R}^n}$ there is a Lebesgue
measurable function $\psi:(0, \varepsilon_0)\rightarrow (0,\infty)$
such that
\begin{equation}\label{eq3.7.3}
0<I(\varepsilon,
\varepsilon_0):=\int\limits_{\varepsilon}^{\varepsilon_0}\psi(t)\,dt
< \infty\,,
\end{equation}
while there exists a function $\alpha=\alpha(\varepsilon,
\varepsilon_0)\geqslant 0$ such that
\begin{equation} \label{eq3.7.4}
\int\limits_{A(y_0, \varepsilon, \varepsilon_0)}
Q(y)\cdot\psi^{\,p}(|y-y_0|)\,dm(y)= \alpha(\varepsilon,
\varepsilon_0)\cdot I^p(\varepsilon, \varepsilon_0)\,,\end{equation}
where $A(y_0, \varepsilon, \varepsilon_0)$ is defined in
(\ref{eq1**}).}
\end{lemma}

\medskip
\begin{proof}
Let us prove the lemma by contradiction. Let us assume that its
conclusion is not true. Then there exists
$\widetilde{\varepsilon}_0>0$ such that for any $k\in {\Bbb N}$
there is a continuum $C_k\subset K$ and a mapping $f_k\in
\frak{F}^{p, \delta}_{Q, a, b}(D)$ such that $h(C_k)\geqslant
\widetilde{\varepsilon}_0,$ however, $h(f_k(C_k))<\frac{1}{k}.$

Due to the compactness of $K,$ there is a sequence $x_k\in C_k$
which convergence to some point $x_0\in K$ as $k\rightarrow \infty.$
Since $K$ is a compactum in $D,$ $x_0 \in D.$ Now we choose a ball
$B(x_0, \varepsilon_1)\subset D,$ where
\begin{equation}\label{eq1E}
\varepsilon_1<\{{\rm dist}\,(x_0, \partial D),
\frac{1}{2}\widetilde{\varepsilon}_0\}\,.
\end{equation}
Observe that
\begin{equation}\label{eq1F}
C_k\cap B(x_0, \varepsilon_1)\ne\varnothing \ne B(x_0,
\varepsilon_1)\setminus C_k
\end{equation}
for sufficiently large $k\in {\Bbb N}.$ Indeed, since $x_k\in C_k$
and $x_k\rightarrow x_0,$ then $C_k\cap B(x_0,
\varepsilon_1)\ne\varnothing$ for sufficiently large $k.$ Moreover,
the inclusion $C_k\subset B(x_0, \varepsilon_1)$ is not possible,
because in this case it may be $d(C_k)\leqslant d(B(x_0,
\varepsilon_1))=2\varepsilon_1$ but, on the other hand,
$d(C_k)\geqslant \widetilde{\varepsilon}_0$ and $d(B(x_0,
\varepsilon_1))=2\varepsilon_1<\widetilde{\varepsilon}_0,$
see~(\ref{eq1E}).

\medskip
Let $M_k$ be a component of $C_k\cap \overline{B(x_0,
\varepsilon_1)}$ consisting the point $x_k.$ By~(\ref{eq1F}) and due
to \cite[Theorem~1.III.5.47]{Ku} $C_k\cap S(x_0,
\varepsilon_1)\ne\varnothing$ for any $k\in {\Bbb N}.$ Let $y_k\in
C_k\cap S(x_0, \varepsilon_1).$ Now, by the triangle inequality
\begin{equation}\label{eq4A}
d(M_k)\geqslant |x_k-y_k|\geqslant
|y_k-x_0|-|x_0-x_k|=\varepsilon_1-|x_0-x_k|\geqslant
\varepsilon_1/2\end{equation}
for sufficiently large $k\in {\Bbb N}$ because $|x_0-x_k|\rightarrow
0$ as $k\rightarrow\infty.$ Let us fix $\varepsilon^*>0$ such that
$0<\varepsilon_1<\varepsilon^{\,*}$ and the ball $B(x_0,
\varepsilon^{\,*})$ lies in $D$ else. Fix a nondegenerate continuum
$F$ in $B(x_0, \varepsilon_1).$ Applying Propositions~\ref{pr_2} and
\ref{pr1X} for the ball $B(x_0, \varepsilon^{\,*})=X=B(x, R),$
$R=\varepsilon^{\,*},$ $x=x_0,$ $E=M_k$ and $F=F,$ and
using~(\ref{eq4A}), we obtain that
\begin{equation}\label{eq3A}
M_p(\Gamma(M_k, F, B(x_0, \varepsilon^{\,*})))\geqslant \frac{1}{C}
\cdot\frac{\min\{\varepsilon_1/2, {\rm
diam}\,F\}}{(\varepsilon^{\,*})^{1+p-n}}:=\Delta=const
\end{equation}
for some a number $\Delta>0$ and for all sufficiently large $k.$
Since $\Gamma(M_k, F, B(x_0, \varepsilon^{\,*}))\subset \Gamma(C_k,
F, D),$ it follows from (\ref{eq3A}) that
\begin{equation}\label{eq4B}
M_p(\Gamma(C_k, F, D)\geqslant \Delta
\end{equation}
for sufficiently large $k.$

\medskip
On the other hand, let us join the points $a$ and $b$ with a path
$\gamma:[0, 1]\rightarrow D,$ $\gamma(0)=a,$ $\gamma(1)=b,$ in $D.$
It follows from the conditions of the lemma that
$h(\bigl|f_k(\gamma)\bigr|)\geqslant \delta$ for any $k=1,2,\ldots
,$ $f_k\in \frak{F}^{p, \delta}_{Q, a, b}(D)$ Due to the compactness
of $\overline{{\Bbb R}^n},$ we may assume that $y_k\rightarrow
y_0\in\overline{{\Bbb R}^n}$ as $k\rightarrow\infty$ for some
sequence $y_k\in f_k(C_k)$ and some $y_0\in \overline{{\Bbb R}^n}.$
Let us firstly consider the case when $y_0\ne\infty.$ We may
consider that the sequences $f_k(a)$ and $f_k(b)$ converge to some
points $z_1$ and $z_2$ as $k\rightarrow \infty$ because
$\overline{{\Bbb R}^n}$ is a compact space. Due to the condition
$h(f_k(a), f_k(b))\geqslant\delta,$ at least one of the above points
does not coincide with $y_0.$ Without loss of generality, we may
consider that $z_1\ne y_0.$ Since $f_k$ is equicontinuous at $a,$
given $\sigma>0$ there is $\chi=\chi(\sigma)>0$ such that $h(f_k(x),
f_k(a))<\sigma$ for $|x-a|<\chi.$ We may chose numbers $r_1, r_2>0$
such small that
\begin{equation}\label{eq6A}
B_h(z_1, r_1)\cap B_h(y_0, r_2)=\varnothing\,.
\end{equation}
For instance, we may set $r_1=r_2=\frac{1}{2}h(z_1, y_0);$ in this
case, the relation (\ref{eq6A}) follows by the triangle inequality.
Again, by the triangle inequality
$$h(f_k(x), z_1)\leqslant h(f_k(x), f_k(a))+ h(f_k(a), z_1)<\sigma+h(f_k(a), z_1)$$
for $|x-a|<\chi$ and since $h(f_k(a), z_1)\rightarrow 0$ as
$k\rightarrow\infty$ the latter relation implies that $f_k(x)\in
B_h(z_1, r_1)$ for sufficiently large $k$ and choosing
$\sigma=r_1/2.$ Let $E=\{|x-a|\leqslant\chi\},$ where $\chi$ is
mentioned above.

\medskip
We set $\Gamma_k=\Gamma(C_k, E, D).$ By~(\ref{eq4B}) and by
Theorem~\ref{th4}
\begin{equation}\label{eq5A}
M_p(\Gamma_k)\geqslant \Delta_2>0
\end{equation}
for sufficiently large $k=1,2,\ldots $ and some $\Delta_2>0.$ Let us
consider $\varepsilon_0$ from the conditions of the lemma. Reducing
it, if necessary, we may consider that $B(y_0, \varepsilon_0)\subset
B_h(y_0, r_2).$

Since $h(f_k(C_k))<1/k,$ we may assume that
\begin{equation*}\label{eq3D}
f_{k}(C_{k})\subset B(y_0, 1/k)\,, \qquad k=1,2,\ldots\,.
\end{equation*}
Let $k_0\in {\Bbb N}$ be such that $B(y_0, 1/k)\subset B(y_0,
\varepsilon_0)$ for all $k\geqslant k_0.$
In this case, observe that
\begin{equation}\label{eq7E}
f_k(\Gamma_k)>\Gamma(S(y_0, 1/k), S(y_0, \varepsilon_0), A(y_0,
1/k,\varepsilon_0))
\end{equation}
that is may be proved similarly to the relation~(\ref{eq3G}). It
follows from~(\ref{eq7E}) that
\begin{equation}\label{eq3I}
\Gamma_k>\Gamma_{f_k}(y_0, 1/k, \varepsilon_0)\,.
\end{equation}
We set $$\eta_{k}(t)=\left\{
\begin{array}{rr}
\psi(t)/I(1/k, \varepsilon_0), & t\in (1/k, \varepsilon_0)\,,\\
0,  &  t\not\in (1/k, \varepsilon_0)\,,
\end{array}
\right. $$
where $I(1/k, \varepsilon_0)=\int\limits_{1/k}^{\varepsilon_0}\,\psi
(t)\, dt.$ Observe that
$\int\limits_{1/k}^{\varepsilon_0}\eta_{k}(t)\,dt=1.$ Then by the
relations~(\ref{eq3.7.4}) and~(\ref{eq3I}), due to the definition of
$f_k$ in~(\ref{eq2*A}) we obtain that
\begin{gather*}M_p(\Gamma_k)\leqslant M_p(\Gamma_{f_k}(y_0, 1/k,
\varepsilon_0))\\
\leqslant \frac{1}{I^p(1/k, \varepsilon_0)}\int\limits_{A(y_0, 1/k,
\varepsilon_0)} Q(y)\cdot\psi^{\,p}(|y-y_0|)\,dm(y)\rightarrow
0\quad \text{as}\quad k\rightarrow\infty\,. \end{gather*}
The latter contradicts with~(\ref{eq5A}). Lemma is proved for the
case $y_0\ne\infty.$

Let us consider the case $y_0=\infty.$ Applying the inversion
$\psi(y)=\frac{y}{|y|^2},$ we consider the family of mappings
$\widetilde{f}_k:=\psi\circ f_k.$ Now, due to the compactness of
$\overline{{\Bbb R}^n},$ we may assume that $y_k\rightarrow
0\in{\Bbb R}^n$ as $k\rightarrow\infty$ for some sequence $y_k\in
\widetilde{f}_k(C_k)$ while $\widetilde{f}_k$ satisfy the
relation~(\ref{eq2*A}) for $y_0=0$ with
$\widetilde{Q}(y)=Q\left(\frac{y}{|y|^2}\right).$ Then the proof of
the lemma made for $f_k,$ $k=1,2,\ldots, $ repeats for mappings
$\widetilde{f}_k,$ $k=1,2,\ldots ,$ by the scheme given
above.~$\Box$ \end{proof}

\medskip
{\it Proof of Theorem~\ref{th2}} immediately follows by
Lemma~\ref{lem4} and Proposition~\ref{pr6}.~$\Box$

\medskip
\begin{example}
Let $f_m(x)=mx\,,$ $m=1,2,\ldots ,$ $f_m:{\Bbb R}^n\rightarrow {\Bbb
R}^n.$ Observe that, $f_m$ satisfies the
relations~(\ref{eq2*A})--(\ref{eqA2}) with $Q\equiv 1$ and $p=n,$
see remarks made in the Introduction. Note that, the family
$\{f_m\}_{m=1}^{\infty}$ is not equicontinuous at the origin, and
$\{f_m\}_{m=1}^{\infty}$ is equicontinuous at any another point
$b\in {\Bbb R}^n.$ Observe that, the relation $h(f_m(a),
f_m(b))\geqslant \delta$ with some $\delta>0$ and all $m=1,2\ldots$
may hold only for $a=0$ and $b\ne 0.$ However,
$\{f_m\}_{m=1}^{\infty}$ does not satisfy the condition 2) of
Theorem~\ref{th1} and the condition 1) of Theorem~\ref{th2} because
$\{f_m\}_{m=1}^{\infty}$ is not equicontinuous at $a=0.$ Obviously,
the conclusions of these theorems do not hold for this family of
mappings.
\end{example}

\medskip The given example indicates that in Theorems~\ref{th1} and~\ref{th2} it is
impossible, generally speaking, to get rid of the condition: ``the
family $\frak{F}^{p, \delta}_{Q, a, b}(D)$ is equicontinuous at the
points $a$ and $b$''. Moreover, this same example indicates that,
generally speaking, one cannot make do with just one such point,
say, point $a,$ since this is still not sufficient for a positive
conclusion. On the other hand, it is obvious that condition $h(f(a),
f(b))\geqslant \delta$ cannot be discarded in the definition of the
class $\frak{F}^{p, \delta}_{Q, a, b}(D).$

\section{Proof of Theorem~\ref{th3}}

We may consider that $y_0=0,$ $\delta(y_0)=1$ and $D={\Bbb
B}^n=\{x\in{\Bbb R}^n: |x|<1\}.$ Let us firstly consider the case
$p=n.$ Define a sequence of homeomorphisms $g_m:{\Bbb
B}^n\rightarrow {\Bbb B}^n,$ $g_m({\Bbb B}^n)={\Bbb B}^n$ by
$$g_m(x)=\frac{x}{|x|}\,\rho_m(|x|)\,,\qquad g_m(0):=0\,,$$
where
%
$$\rho_m(r)= \exp\left\{-\int\limits_{r}^1\frac{dt}{tq_{0,
m}^{1/(n-1)}(t)}\right\}\,, \qquad q_{0,
m}(r):=\frac{1}{\omega_{n-1}r^{n-1}}\int\limits_{|x|=r}Q_m(x)\,dS\,,$$
%
%
$$Q_m(x)\quad=\quad \left \{\begin{array}{rr} Q(x) , & \ |x|> 1/m\ ,
\\ 1\ ,  &  |x|\leqslant 1/m\,.
\end{array} \right.$$
%
Set
%
$$\rho(r)=
\exp\left\{-\int\limits_{r}^1\frac{dt}{tq_{0}^{1/(n-1)}(t)}\right\}\,,
\qquad
q_{0}(r):=\frac{1}{\omega_{n-1}r^{n-1}}\int\limits_{|x|=r}Q(x)\,dS\,.
$$
%
Since $Q$ is locally integrable, $q_0(t)\ne \infty$ for almost all
$t>0.$ Consequently, the function $\rho$ is strictly monotone; in
particular, $\rho^{\,-1}(r)$ is well-defined. Let
$f_m(x):=g^{\,-1}_m(x).$ We may calculate directly that
$$f_m(x)=\begin{cases}\frac{x}{m}\exp\{I_m\}\,,&|x|\leqslant \exp\{-I_m\}\,,
\\ \frac{x}{|x|}\rho^{\,-1}(|x|)\,, &\exp\{-I_m\}\leqslant |x|<1\end{cases}\,,$$
where
$I_m:=\int\limits_{\frac{1}{m}}^1\frac{dt}{tq^{\frac{1}{n-1}}_0(t)}.$
Observe that, the sequence $f_m$ satisfies the
relations~(\ref{eq2*A})--(\ref{eqA2}) at the origin, that is proved
in \cite[Theorem~3.10]{Sev$_1$}. At the same time, the sequence
$f_m$ converges to the mapping $f$ uniformly in ${\Bbb B}^n,$ where
$f$ is defined as
$$f(x)=\begin{cases}0\,,&|x|\leqslant \exp\{-I_0\}\,,
\\ \frac{x}{|x|}\rho^{\,-1}(|x|), &\exp\{-I_0\}\leqslant |x|<1\end{cases}\,,$$
where $I_0:=\int\limits_{0}^{\delta(y_0)}
\frac{dt}{tq_{y_0}^{\frac{1}{n-1}}(t)}<\infty$ by the assumption.
So, if we take a sequence of continua $C_m\equiv C,$ where $C$ is a
fixed continuum lying in the ball $B(0, \exp\{-I_0\}),$ we observe
that $h(f_m(C))\rightarrow 0$ as $m\rightarrow \infty.$ By the
construction of $f_m,$ they fix infinitely many points in the ring
$\exp\{-I_0\}\leqslant |x|<1,$ in particular, the family $\{f_m\}$
is equicontinuous at these points.

\medskip
Let us consider the similar construction for $p\ne n.$ Define a
sequence of homeomorphisms $g_m:{\Bbb B}^n\rightarrow {\Bbb B}^n,$
$g_m({\Bbb B}^n)={\Bbb B}^n$ by
$$g_m(x)=\frac{x}{|x|}\,\rho_m(|x|)\,,\qquad g_m(0):=0\,,$$
where
%
$$\rho_m(r)=\left(1+\frac{n-p}{p-1}\int\limits_{r}^{1}\frac{dt}{t^{\frac{n-1}{p-1}}\,
q_{0, m}^{\frac{1}{p-1}}(t)}\right)^{\frac{p-1}{p-n}}\,, \qquad
q_{0,
m}(r):=\frac{1}{\omega_{n-1}r^{n-1}}\int\limits_{|x|=r}Q_m(x)\,dS\,,$$
%
%
$$Q_m(x)\quad=\quad \left \{\begin{array}{rr} Q(x) , & \ |x|> 1/m\ ,
\\ 1\ ,  &  |x|\leqslant 1/m\,.
\end{array} \right.$$
%
Set
%
$$\rho(r)=
\left(1+\frac{n-p}{p-1}\int\limits_{r}^{1}\frac{dt}{t^{\frac{n-1}{p-1}}\,
q_{0}^{\frac{1}{p-1}}(t)}\right)^{\frac{p-1}{p-n}}\,, \qquad
q_{0}(r):=\frac{1}{\omega_{n-1}r^{n-1}}\int\limits_{|x|=r}Q(x)\,dS\,.
$$
%
Since $Q$ is locally integrable, $q_0(t)\ne \infty$ for almost all
$t>0.$ Consequently, the function $\rho$ is strictly monotone; in
particular, $\rho^{\,-1}(r)$ is well-defined. Let
$f_m(x):=g^{\,-1}_m(x).$ We may calculate directly that
$$f_m(x)=\begin{cases}\frac{x}{|x|}\left(r^{\frac{p-n}{p-1}}-1+
\frac{1}{m^{\frac{p-n}{p-1}}}-\frac{n-p}{p-1}I_m\right)^{\frac{p-1}{p-n}}
\,,&|x|\leqslant J_m\,,
\\ \frac{x}{|x|}\rho^{\,-1}(|x|)\,, &J_m\leqslant |x|<1\end{cases}\,,$$
where $I_m:=\int\limits_{\frac{1}{m}}^1\frac{dt}{t^{\frac{n-1}{p-1}}
q^{\frac{1}{p-1}}_0(t)}$ and
$J_m:=\left(1+\frac{n-p}{p-1}\int\limits_{1/m}^{1}\frac{dt}{t^{\frac{n-1}{p-1}}\,
q_{0}^{\frac{1}{p-1}}(t)}\right)^{\frac{p-1}{p-n}}.$
Observe that, the sequence $f_m$ satisfies the
relations~(\ref{eq2*A})--(\ref{eqA2}) at the origin, that is proved
in \cite[Theorem~4]{SalSev}. At the same time, the sequence $f_m$
converges to the mapping $f$ uniformly in ${\Bbb B}^n,$ where $f$ is
defined as
$$f(x)=\begin{cases}0\,,&|x|\leqslant J_0\,,
\\ \frac{x}{|x|}\rho^{\,-1}(|x|), &J_0\leqslant |x|<1\end{cases}\,,$$
where
$J_0:=\left(1+\frac{n-p}{p-1}\int\limits_{0}^{1}\frac{dt}{t^{\frac{n-1}{p-1}}\,
q_{0}^{\frac{1}{p-1}}(t)}\right)^{\frac{p-1}{p-n}}>0$ because by the
assumption $\int\limits_{0}^{1}\frac{dt}{t^{\frac{n-1}{p-1}}\,
q_{0}^{\frac{1}{p-1}}(t)}<\infty.$ So, if we take a sequence of
continua $C_m\equiv C,$ where $C$ is a fixed continuum lying in the
ball $B(0, J_0),$ we observe that $h(f_m(C))\rightarrow 0$ as
$m\rightarrow \infty.$ By the construction of $f_m,$ they fix
infinitely many points in the ring $J_0\leqslant |x|<1,$ in
particular, the family $\{f_m\}$ is equicontinuous at these
points.~$\Box$

\section{On Koebe-Bloch theorem and Orlicz-Sobolev classes}

Let $D\subset {\Bbb R}^n,$ $f:D\rightarrow {\Bbb R}^n$ be a discrete
open mapping, $\beta: [a,\,b)\rightarrow {\Bbb R}^n$ be a path, and
$x\in\,f^{\,-1}(\beta(a)).$ A path $\alpha: [a,\,c)\rightarrow D$ is
called a {\it maximal $f$-lifting} of $\beta$ starting at $x,$ if
$(1)\quad \alpha(a)=x\,;$ $(2)\quad f\circ\alpha=\beta|_{[a,\,c)};$
$(3)$\quad for $c<c^{\prime}\leqslant b,$ there is no a path
$\alpha^{\prime}: [a,\,c^{\prime})\rightarrow D$ such that
$\alpha=\alpha^{\prime}|_{[a,\,c)}$ and $f\circ
\alpha^{\,\prime}=\beta|_{[a,\,c^{\prime})}.$ Similarly, we may
define a maximal $f$-lifting $\alpha: (c,\,b]\rightarrow D$ of a
path $\beta: (a,\,b]\rightarrow {\Bbb R}^n$ ending at
$x\in\,f^{\,-1}(\beta(b)).$ The following assertion holds
(see~\cite[Lemma~3.12]{MRV}).

\medskip
\begin{proposition}\label{pr3A}
{\it Let $f:D\rightarrow {\Bbb R}^n,$ $n\geqslant 2,$ be an open
discrete mapping, let $x_0\in D,$ and let $\beta: [a,\,b)\rightarrow
{\Bbb R}^n$ be a path such that $\beta(a)=f(x_0)$ and such that
either $\lim\limits_{t\rightarrow b}\beta(t)$ exists, or
$\beta(t)\rightarrow \partial f(D)$ as $t\rightarrow b.$ Then
$\beta$ has a maximal $f$-lifting $\alpha: [a,\,c)\rightarrow D$
starting at $x_0.$ If $\alpha(t)\rightarrow x_1\in D$ as
$t\rightarrow c,$ then $c=b$ and $f(x_1)=\lim\limits_{t\rightarrow
b}\beta(t).$ Otherwise $\alpha(t)\rightarrow \partial D$ as
$t\rightarrow c.$}
\end{proposition}

\medskip
The version of the following result was obtained by the second
co-author in~\cite{ST$_1$}. Now we obtain some more general form of
it as a corollary from Theorem~\ref{th2}.

\medskip
\begin{theorem}\label{th5}\,(Koebe-Bloch theorem for $p$-modulus, cf.~\cite{ST$_1$}).
{\it\,Assume that, any $f\in\frak{F}^{p, \delta}_{Q, a, b}(D)$ is
open and discrete, besides that, all the conditions of
Theorem~\ref{th2} are satisfied. Then a family $\frak{F}^{p,
\delta}_{Q, a, b}(D)$ is uniformly open on every compactum $K,$
i.e., for every $\varepsilon_0>0$ there exists $r_0=r_0(K,
\varepsilon_0)>0$ such that $B_h(f(x_0), r_0)\subset f(B(x_0,
\varepsilon_0))$ for every $f\in \frak{F}^{p, \delta}_{Q, a, b}(D)$
and every $B(x_0, \varepsilon_0)\subset K,$ where $B_h(f(x_0),
r_0)=\{y\in\overline{{\Bbb R}^n}: h(y, f(x_0))<r_0\}.$}
\end{theorem}

\medskip
\begin{proof}
Mainly we apply the arguments used in \cite{ST$_1$} Assume the
contrary, i.e., there exists a compactum $K$ in $D$ such that
$\frak{F}^{p, \delta}_{Q, a, b}(D)$ is not uniformly open on $K.$
Then there exists $\varepsilon_0>0$ such that for any $m\in {\Bbb
N}$ there exists $x_m\in K$ and $f_m\in \frak{F}^{p, \delta}_{Q, a,
b}(D)$ such that $B(x_m, \varepsilon_0)\subset K$ and $B_h(f_m(x_m),
1/m)\setminus f_m(B(x_m, \varepsilon_0))\ne\varnothing.$ Let $y_m\in
B_h(f_m(x_m), 1/m)\setminus f_m(B(x_m, \varepsilon_0)).$ We may
consider that $f_m(x_m)$ and $y_m$ converge to some point
$\omega_{*}$ as $m\rightarrow \infty.$ We may consider
$\omega_{*}\ne\infty,$ otherwise we consider
$\widetilde{f}_m:=\psi\circ f_m,$ $\psi(x)=\frac{x}{|x|^2},$ instead
$f_m$ follow.

\medskip
Join the points $f_m(x_m)$ and $y_m$ by the segment
$r_m(t)=f_m(x_m)+t(y_m-f_m(x_m)),$ $t\in [0, 1].$ Since $|r_m|\cap
f_m(B(x_m, \varepsilon_0))\ne\varnothing\ne |r_m|\setminus
f_m(B(x_m, \varepsilon_0)),$ by \cite[Theorem~1.I.5, \S46]{Ku} there
is a point $z_m=r_m(t_m)\in
\partial f_m(B(x_m, \varepsilon_0)).$ Without loss of generality, we
may assume that the path $\beta_m:=r_m|_{[0, t_m)}$ lies in
$f_m(B(x_m, \varepsilon_0)).$ Let $\alpha_m$ be a maximal
$f_m$-lifting of $\beta_m$ starting at $x_m$ (it exists by
Proposition~\ref{pr3A}). By the same Proposition either
$\alpha_m(t)\rightarrow x_1\in B(x_m, \varepsilon_0)$ as
$t\rightarrow c_m-0$ (in this case, $c_m=1$ and $f_m(x_1)=y_m$), or
$\alpha_m(t)\rightarrow S(x_m, \varepsilon_0)$ as $t\rightarrow
c_m.$ Observe that, the first situation is excluded. Indeed, if
$f_m(x_1)=y_m,$ then $y_m\in f_m(B(x_m, \varepsilon_0)),$ that
contradicts the choice of $y_m.$ Thus, $\alpha_m(t)\rightarrow
S(x_m, \varepsilon_0)$ as $t\rightarrow c_m.$ Observe that,
$\overline{|\alpha_m|}$ is a continuum in $\overline{B(x_m,
\varepsilon_0)}$ and $d(\overline{|\alpha_m|})\geqslant d(x_m,
S(x_m, \varepsilon_0))=\varepsilon_0.$ Since
$\overline{|\alpha_m|}\subset \overline{B(x_m,
\varepsilon_0)}\subset K$ and $K$ is a continuum, it follows that
$h(\overline{|\alpha_m|})\geqslant \varepsilon^{\,*}_0$ for
$m=1,2,\ldots$ and some $\varepsilon^{\,*}_0>0,$ as well. On the
other hand, $f_m(|\alpha_m|)\subset |\beta_m|$ and
$h(|\beta_m|)\leqslant|f_m(x_m)-y_m|<1/m,$ $m\in {\Bbb N}.$ Since
$f_m(x_m)$ and $y_m$ converge to some point $\omega_{*}$ as
$m\rightarrow \infty$ and $\omega_{*}\ne\infty,$ we have that
$h(|\beta_m|)\rightarrow 0$ as $m\rightarrow\infty,$ as well. The
latter contradicts the statement of Theorem~\ref{th2}. The obtained
contradiction proves the theorem.~$\Box$
\end{proof}

\medskip
We set
\begin{gather*}l\left(f^{\,\prime}(x)\right)\,=\,\min\limits_{h\in {\Bbb R}^n
\backslash \{0\}} \frac {|f^{\,\prime}(x)h|}{|h|}\,,\quad
J(x,f)=\det f^{\,\prime}(x)\,,\\
K_{I}(x,f)\quad =\quad\left\{
\begin{array}{rr}
\frac{|J(x,f)|}{{l\left(f^{\,\prime}(x)\right)}^n}, & J(x,f)\ne 0,\\
1,  &  f^{\,\prime}(x)=0, \\
\infty, & {\rm otherwise}
\end{array}
\right.\,.\end{gather*}
Let $\varphi:[0,\infty)\rightarrow[0,\infty)$ be a non-decreasing
function, $f$ be a locally integrable vector function of $n$ real
variables $x_1,\ldots,x_n,$ $f=(f_1,\ldots,f_n),$ $f_i\in
W_{loc}^{1,1},$ $i=1,\ldots,n.$ We say that $f:D\rightarrow {\Bbb
R}^n$ belongs to the class $W^{1,\varphi}_{loc},$ we write $f\in
W^{1,\varphi}_{loc},$ if $$\int\limits_{G}\varphi\left(|\nabla
f(x)|\right)\,dm(x)<\infty$$ for any compact subdomain of $G\subset
D,$ where $|\nabla
f(x)|=\sqrt{\sum\limits_{i=1}^n\sum\limits_{j=1}^n\left(\frac{\partial
f_i}{\partial x_j}\right)^2}.$ Let $f:D\rightarrow D^{\,\prime},$
$x_0\in D$ and $y_0=f(x_0).$ Assume that, $f$ is a homeomorphism and
denote by $g:=f^{\,-1}.$ Observe that
$$
g(\Gamma(S(y_0, r_1)), S(y_0, r_2), f(D))=\Gamma_f(y_0, r_1, r_2)\,.
$$
Taking into account the above and applying \cite[Theorem~2.2]{KR},
cf.~\cite[Theorem~5.6]{Sev$_2$}, we obtain the following.

\medskip
\begin{proposition}\label{pr2B}
{\it\, Let $g:D^{\,\prime}\rightarrow D,$ $n\geqslant 3,$ be a
homeomorphism in $W_{\rm loc}^{1, \varphi}(D^{\,\prime})$ with
$K_I(y, g)\in L^1_{\rm loc}(D^{\,\prime}),$ where $K_I(x, f)$ is
defined above, and let $\varphi:[0,\infty)\rightarrow[0,\infty)$ be
a non-decreasing function. Assume that
\begin{equation}\label{eqOS3.0a_1}
\int\limits_{1}^{\infty}\left[\frac{t}{\varphi(t)}\right]^
{\frac{1}{n-2}}dt<\infty\,.
\end{equation}
Then $f=g^{\,-1}:D\rightarrow D^{\,\prime},$ $D^{\,\prime}=f(D),$
satisfies (\ref{eq2*A})--(\ref{eqA2}) for every $y_0\in {\Bbb R}^n$
and $0<r_1<r_2<d_0=\sup\limits_{y\in D^{\,\prime}}|y-y_0|,$ where
$Q=K_I(y, g).$}
\end{proposition}

\medskip
Let us also formulate the consequence of Theorems~\ref{th1},
\ref{th2} and~\ref{th5} for Orlicz-Sobolev classes. Given $a, b\in
D,$ $a\ne b,$ a Lebesgue measurable function $Q:D\rightarrow [0,
\infty],$ a non-decreasing function
$\varphi:[0,\infty)\rightarrow[0,\infty)$ and $\delta>0$ we define
the family $\frak{OS}^{\varphi, \delta}_{Q, a, b}(D)$ of all
homeomorphisms $f:D\rightarrow {\Bbb R}^n,$ $n\geqslant 3,$ such
that $g:=f^{\,-1}$ belongs to $W_{\rm loc}^{1, \varphi}(f(D)),$
$h(f(a), f(b))\geqslant \delta$ and $K_I(y, f^{\,-1})\leqslant Q(y)$
for almost all $y\in f(D).$ The following statement is true.

\medskip
\begin{theorem}\label{th1A} {\it Let $D$ be a domain in ${\Bbb R}^n,$
$n\geqslant 3,$ let $a, b\in D,$ $a\ne b,$ let $\delta>0$ and let
$Q:D\rightarrow [0, \infty]$ be a Lebesgue measurable function.
Assume that the following conditions hold:

\medskip
1) the domain $D$ is $p$-uniform,

2) the family $\frak{OS}^{\varphi, \delta}_{Q, a, b}(D)$ is
equicontinuous at the points $a$ and $b,$

3) $\varphi$ satisfies the Calderon condition~(\ref{eqOS3.0a_1}),

4) at least one of the conditions $3_1)$--$3_2)$ in
Theorem~\ref{th1} holds.

\medskip
Then the following holds: given $\varepsilon>0$ there is
$\delta_1(\varepsilon)>0$ such that $h(f(C))\geqslant \delta_1$ for
any $f\in \frak{OS}^{\varphi, \delta}_{Q, a, b}(D)$ and any
continuum $C\subset D$ with $h(C)\geqslant \varepsilon.$  }
\end{theorem}

\medskip
{\it\, Proof} directly follows from Theorem~\ref{th1} and
Proposition~\ref{pr2B}.~$\Box$

\medskip
\begin{theorem}\label{th2A} {\it Let $D$ be a domain in ${\Bbb R}^n,$
$n\geqslant 3,$ let $a, b\in D,$ $a\ne b,$ let $\delta>0$ and let
$Q:D\rightarrow [0, \infty]$ be a Lebesgue measurable function.
Assume that the following conditions hold:

\medskip
1) the family $\frak{OS}^{\varphi, \delta}_{Q, a, b}(D)$ is
equicontinuous at the points $a$ and $b,$

2) $\varphi$ satisfies the Calderon condition~(\ref{eqOS3.0a_1}),

3) at least one of the conditions $2_1)$--$2_2)$ in
Theorem~\ref{th1} holds.

\medskip
Then the following holds: given $\varepsilon>0$ there is
$\delta_1(\varepsilon)>0$ such that $h(f(C))\geqslant \delta_1$ for
any $f\in \frak{OS}^{\varphi, \delta}_{Q, a, b}(D)$ and any
continuum $C\subset K$ with $h(C)\geqslant \varepsilon.$  }
\end{theorem}

\medskip
{\it\, Proof} directly follows from Theorem~\ref{th2} and
Proposition~\ref{pr2B}.~$\Box$

\medskip
\begin{theorem}\label{th5A}\,(Koebe-Bloch theorem for Orlicz-Sobolev classes).
{\it\,Assume that, all the conditions of Theorem~\ref{th2A} are
satisfied. Then a family $\frak{OS}^{\varphi, \delta}_{Q, a, b}(D)$
is uniformly open on every compactum $K$ in $D,$ i.e., for every
$\varepsilon_0>0$ there exists $r_0=r_0(K, \varepsilon_0)>0$ such
that $B_h(f(x_0), r_0)\subset f(B(x_0, \varepsilon_0))$ for every
$f\in \frak{OS}^{\varphi, \delta}_{Q, a, b}(D)$ and every $B(x_0,
\varepsilon_0)\subset K,$ where $B_h(f(x_0),
r_0)=\{y\in\overline{{\Bbb R}^n}: h(y, f(x_0))<r_0\}.$}
\end{theorem}

\medskip
{\it\, Proof} directly follows from Theorem~\ref{th5} and
Proposition~\ref{pr2B}.~$\Box$

\medskip
The paper is published in the preprint form, see \cite{RS}.

\medskip
{\bf Declarations.}

\medskip
{\bf Conflicts of interest.} The author has no financial or
proprietary interests in any material discussed in this article.

\medskip
{\bf Availability of data and material.} The datasets generated
and/or analysed during the current study are available from the
corresponding author on reasonable request.

\medskip
{\bf Acknowledgements.} The work was supported by the National
Research Foundation of Ukraine (Project ``Analogues of
Carath\'{e}odory and Koebe-Bloch theorems for Orlycz-Sobolev
classes'', Project number 2025.02/0010).

\medskip
{\bf CONTACT INFORMATION}

\medskip
{\bf \noindent Denys Romash} \\
Zhytomyr Ivan Franko State University,  \\
40 Velyka Berdychivs'ka Str., 10 008  Zhytomyr, UKRAINE \\
dromash@num8erz.eu

\medskip
{\bf \noindent Evgeny Sevost'yanov} \\
{\bf 1.} Zhytomyr Ivan Franko State University,  \\
40 Velyka Berdychivs'ka Str., 10 008  Zhytomyr, UKRAINE \\
{\bf 2.} Institute of Applied Mathematics and Mechanics\\
of NAS of Ukraine, \\
19 Henerala Batyuka Str., 84 116 Slov'yansk,  UKRAINE\\
esevostyanov2009@gmail.com

\end{document}